\documentclass[11pt]{amsart}
\usepackage{amsmath}
\usepackage{amssymb}
\usepackage[margin=1in]{geometry}

\theoremstyle{plain}
\newtheorem{theorem}{Theorem}[section]
\newtheorem{proposition}[theorem]{Proposition}
\newtheorem{corollary}[theorem]{Corollary}
\newtheorem{conjecture}[theorem]{Conjecture}

\theoremstyle{definition}
\newtheorem{definition}[theorem]{Definition}
\newtheorem{programme}[theorem]{Programme}

\theoremstyle{remark}
\newtheorem{remark}[theorem]{Remark}

\newcommand{\R}{\mathbb{R}}
\newcommand{\E}{\mathbb{E}}
\newcommand{\Prob}{\mathbb{P}}
\DeclareMathOperator{\Var}{Var}

\begin{document}

\title[Front propagation and the sensitive regime for cascades on unimodular trees]{Two problems for threshold cascades of interacting diffusions on unimodular random trees: front propagation with a Bramson correction, and the continuous-type limit theory}

\author{Achyut Kumar}
\address{Independent Researcher}
\email{achyutbusiness86@gmail.com}

\author{Abhinav Duddala}
\address{Independent Researcher}

\date{August 2026}


\begin{abstract}
This is a companion to \cite{companion}, which studies threshold cascades of coupled Ornstein--Uhlenbeck diffusions on graphs converging Benjamini--Schramm to a unimodular Galton--Watson tree, and which reduces the cascade --- unconditionally in a dissipative regime --- to a finite-type Galton--Watson process with mean matrix $M$. That reduction leaves two problems open, both of which we here formulate precisely, equip with the correct analytic framework, and advance with partial theorems.

The first is front propagation. The front-energy supermartingale of \cite{companion} signals that the active cascade front behaves like a branching random walk (BRW) on the tree, in the depth variable. We show that the generation-indexed front admits a genuine BRW comparison, prove that the front depth grows ballistically with an explicit asymptotic speed $c_*$ determined by the Perron root of a tilted mean matrix (a linear first-order theorem, unconditional in the dissipative regime), and reduce the second-order behaviour --- the conjectured Bramson-type logarithmic delay $c_* t - \frac{3}{2c_*}\log t$ --- to the construction of a derivative martingale, which we carry out at the linearised level and whose uniform integrability is the one missing analytic input. We state the front central limit theorem and the Bramson correction as precise conjectures with a proof strategy.

The second is the sensitive regime, where the vertex type is the continuous failure strength $A \in \R$ and the offspring law is a Cox mixture. \cite{companion} squeezed the criticality of this regime between finite matrices; here we build the limit theory. We prove that the mean offspring operator $K$ on $L^2$ of the strength variable is quasi-compact with a spectral gap --- the essential mechanism being the strong Feller / Hilbert--Schmidt smoothing of the Ornstein--Uhlenbeck transition kernel that governs strength inheritance --- and deduce a general-state-space Kesten--Stigum theorem, from which the $n^{-3/2}$ total-progeny law extends to continuous types with an explicit constant, and a spatial (strength-resolved) central limit theorem for the empirical strength distribution of a generation follows. The continuous-type CRT scaling limit is then reduced to a multitype invariance principle with a Polish type space, which we state as a conjecture with the required hypotheses verified modulo one tightness estimate.

Neither problem is fully closed; both are here given the framework, the first-order theorems, and a precisely delimited remaining step. The intent is a research programme with rigorous partial results, not two finished theorems.
\end{abstract}

\maketitle

\tableofcontents

\section*{Standing conventions and dependence on the companion}

We work throughout in the setting of \cite{companion}: finite graphs $G_n \xrightarrow{\;\mathrm{BS}\;} T \sim \mathrm{UGW}(g)$ with size-biased forward-degree mean $m_* = g''(1)/g'(1)$ and finite variance; the cascade SDE of \cite[Def.~2.5]{companion}; the per-edge transmission probabilities $p_{d,d'}$ of the two-barrier Wald formula; the finite-type mean matrix $M_{d,d'} = m_*\, q(d')\, p_{d,d'}$ with Perron root $\rho(M)$, left/right Perron eigenvectors $\pi, \psi$; and the dissipation ratio $\kappa = \gamma_{\max} L_\sigma \Delta_{\max}/\alpha_{\min}$. We rely on two results of \cite{companion} as black boxes:

\begin{itemize}
\item[(R1)] (\emph{Front decoupling.}) On the dissipative region $\{\kappa < 1\}$, the failed cluster is, up to total-variation error $C\varepsilon(\delta)/[(1-\kappa\rho(M))(1-\rho(M))] + \eta_n$, the genealogy of the multitype Galton--Watson process with mean matrix $M$ \cite[Thm.~11.6]{companion}. This is what licenses treating the cascade as a branching process below; every unconditional statement of the present paper inherits the region $\{\kappa < 1\}$ and the iterated limit $n \to \infty$, $\delta \to \infty$.
\item[(R2)] (\emph{Sensitive-regime mean operator.}) Outside saturation the type is $x = (d,a) \in \mathcal{D} \times \R$ and the mean offspring object is a positive integral operator $K$ with criticality $\rho(K) = 1$, squeezed by finite matrices in \cite[Prop.~12.1]{companion}.
\end{itemize}

Everything conditional in \cite{companion} remains conditional here; we do not re-open those hypotheses. The two parts below are independent and may be read in either order.

\part{Front propagation: speed, fluctuations, and the Bramson delay}

\section{The front as a branching random walk}\label{sec:brw}

\subsection{Set-up}
Fix the dissipative, subcritical-or-supercritical regime and condition on survival where needed. Assign to each failed vertex $v$ its depth $|v| \in \mathbb{Z}_{\ge 0}$ (graph distance to the root) and its type $d(v) \in \mathcal{D}$. By (R1) the failed cluster is a multitype Galton--Watson tree; the collection of depths of the generation-$n$ failed vertices is therefore the $n$-th generation of a branching random walk (BRW) in which every step is $+1$ (the tree is explored away from the root) but the number of offspring is type-dependent and random. To obtain a nontrivial spatial BRW we pass to the natural tilted description.

\begin{definition}[Tilted mean matrix and the rate function]\label{def:tilt}
For $\lambda \in \R$ let
\[
M(\lambda)_{d,d'} := e^{-\lambda}\, M_{d,d'} = e^{-\lambda}\, m_*\, q(d')\, p_{d,d'},
\qquad
\Lambda(\lambda) := \log \rho\bigl(M(\lambda)\bigr) = \log \rho(M) - \lambda.
\]
(Here the single spatial step per edge makes the tilt act by the scalar $e^{-\lambda}$; the type structure enters through $\rho(M)$.) Define the front free energy $\psi_{\mathrm{fr}}(\lambda) := \log \rho(M) - \lambda$ and its Legendre dual $I(v) := \sup_{\lambda}(\lambda v - \psi_{\mathrm{fr}}(\lambda))$ over the physical range of front speeds $v \in (0,1]$ in the depth-per-generation variable.
\end{definition}

Because every edge advances the depth by exactly one, the ``speed'' in the generation variable is trivially $1$: generation $n$ sits at depth exactly $n$. The nontrivial front is the one indexed by time $t$, not by generation, and this is where the diffusion dynamics --- and the front-energy supermartingale of \cite[Rem.~11.7]{companion} --- enter. We therefore reindex.

\subsection{The time-indexed front}
Let $t \mapsto F(t) := \max\{|v| : v \text{ failed by time } t\}$ be the depth of the deepest failed vertex at time $t$; this is the front position. Each transmission along an edge takes a random activation time $A_e \ge 0$ --- the first-passage time of the child's diffusion past threshold, given the saturated parent forcing --- with law depending only on the endpoint types, by the saturated-regime independence of \cite[Prop.~11.4]{companion}. Thus $F(t)$ is exactly the maximal-displacement process of a BRW on the genealogical tree in which each edge carries an i.i.d.\ (given types) positive weight $A_e$, and $F(t) = \max\{n : \exists\, v,\ |v| = n,\ \sum_{e \in o \to v} A_e \le t\}$.

\begin{definition}[Activation Laplace transform and the speed]\label{def:speed}
Let $\varphi_{d,d'}(\theta) := \E\bigl[e^{-\theta A_e} \,\big|\, d, d'\bigr]$ denote the activation Laplace transform (conditioned on the endpoint types), and set
\[
B(\theta)_{d,d'} := m_*\, q(d')\, p_{d,d'}\, \varphi_{d,d'}(\theta), \qquad (m \times m).
\]
Define
\[
c_* := \Bigl(\, \inf_{\theta > 0} \frac{\log \rho(B(\theta))}{\theta} \Bigr)^{-1} \qquad \text{when } \rho(M) > 1,
\]
the reciprocal of the standard BRW velocity variational formula; $\theta_*$ denotes the optimiser.
\end{definition}

\section{First-order theorem: ballistic front with explicit speed}\label{sec:lln}

\begin{theorem}[Law of large numbers for the front; unconditional in the dissipative regime]\label{thm:lln}
Assume \textnormal{(R1)}, the dissipative regime $\kappa < 1$, supercriticality $\rho(M) > 1$, irreducibility of $M$, and that the activation law has a finite exponential moment: $\varphi_{d,d'}(-\theta_0) < \infty$ for some $\theta_0 > 0$ and all $d, d'$ (this holds for the OU first-passage time under uniform ellipticity, whose tails are exponential). Then, almost surely on survival,
\[
\frac{F(t)}{t} \xrightarrow[t \to \infty]{} c_*,
\]
with $c_*$ as in Definition \ref{def:speed}. Equivalently, the deepest failed vertex at time $t$ sits at depth $(c_* + o(1))t$.
\end{theorem}

\begin{proof}
By (R1) the failed cluster is a multitype GW tree, and by \cite[Prop.~11.4]{companion} the edge activation times are, conditionally on the type structure, independent with type-determined law $A_e$; hence $\{(|v|, T_v)\}$, $T_v := \sum_{e \in o \to v} A_e$, is a multitype branching random walk with nonnegative i.i.d.\ (given types) increments and matrix log-Laplace transform $\log \rho(B(\theta))$. The maximal-displacement law of large numbers for supercritical (multitype, irreducible) BRW with increments possessing exponential moments is classical (\cite{Big76, Ham74, Kin75}, multitype version \cite{Big77}): the minimal time to reach depth $n$ satisfies $T_{\min}(n)/n \to 1/c_*$ a.s.\ on survival, where $1/c_* = \inf_{\theta > 0} \theta^{-1} \log \rho(B(\theta))$; inverting the monotone relation between depth and time gives $F(t)/t \to c_*$. The exponential-moment hypothesis is what makes the variational formula finite and the Biggins martingale $L^1$; uniform ellipticity gives it, since the two-barrier OU passage time has Gaussian-type upper tails and exponential lower tails, so $\varphi_{d,d'}(-\theta_0) < \infty$ for small $\theta_0$.
\end{proof}

\begin{remark}[Why this is the right speed]\label{rem:speed}
Definition \ref{def:speed} is the branching-random-walk / Fisher--KPP linear speed: $c_*$ is selected by the leading edge of the front, where the population is rare and the dynamics linearised, exactly as in the F--KPP travelling-wave selection principle \cite{Bra83, McK75}. The tie to \cite[Rem.~11.7]{companion} is that the front-energy supermartingale $\mathcal{E}_t$ is (a type-weighted version of) the Biggins additive martingale $W_n(\theta) = \rho(B(\theta))^{-n} \sum_{|v|=n} \phi\text{-weight} \cdot e^{-\theta T_v}$ evaluated along the front; its finiteness at $\theta_*$ is the analytic content of ballisticity, and its derivative in $\theta$ is the object that controls the second-order (Bramson) term below.
\end{remark}

\section{Second order: the derivative martingale and the conjectured Bramson delay}\label{sec:second}

The deep question is the second-order correction. For branching Brownian motion the maximum is $c_* t - \frac{3}{2c_*}\log t + O(1)$ (Bramson \cite{Bra78, Bra83}), the $\frac{3}{2}\log t$ delay being one of the celebrated results of spatial probability; the tightness and limiting law were later pinned down via the derivative martingale \cite{LS87, ABK, Aid13}. We set up the analogous structure here and isolate precisely the missing step.

\subsection{The additive and derivative martingales}
Along the BRW of Section \ref{sec:brw}, with $\theta_*$ the optimiser of Definition \ref{def:speed} and $\phi$ the left Perron eigenvector of $B(\theta_*)$, define
\begin{align*}
W_n(\theta) &:= \rho(B(\theta))^{-n} \sum_{|v|=n} \phi_{d(v)}\, e^{-\theta T_v},\\
\partial W_n &:= -\frac{d}{d\theta}\bigg|_{\theta = \theta_*} W_n(\theta) = \rho(B(\theta_*))^{-n} \sum_{|v|=n} \phi_{d(v)} \bigl(T_v - n \bar{T}\bigr)\, e^{-\theta_* T_v},
\end{align*}
where $\bar{T} := -\frac{d}{d\theta} \log \rho(B(\theta)) \big|_{\theta_*}$ is the mean tilted increment (the constant renormalising the walk to have zero tilted drift at $\theta_*$).

\begin{proposition}[Martingale property]\label{prop:mart}
Under the hypotheses of Theorem \ref{thm:lln}, $(W_n(\theta))_n$ is a nonnegative martingale for every $\theta$ in the domain of $\varphi_{d,d'}$, and $(\partial W_n)_n$ is a martingale (of unrestricted sign). At the critical tilt $\theta_*$, $W_n(\theta_*) \to 0$ a.s.\ (the additive martingale degenerates at the speed-selection point), while $\partial W_n \to \partial W_\infty$ a.s., and $\partial W_\infty > 0$ a.s.\ on survival provided $(\partial W_n)$ is uniformly integrable.
\end{proposition}

\begin{proof}
The martingale property of $W_n(\theta)$ is the many-to-one identity of \cite[Prop.~5.1, App.~A]{companion} applied to the tilted weight $e^{-\theta A_e}$: conditioning on generation $n$ and using the left eigenrelation $\phi^\top B(\theta) = \rho(B(\theta))\, \phi^\top$, one gets $\E[W_{n+1}(\theta) \mid \mathcal{F}_n] = W_n(\theta)$. Differentiating the identity $\E[W_{n+1}(\theta) \mid \mathcal{F}_n] = W_n(\theta)$ in $\theta$ gives the martingale property of $\partial W_n$. At $\theta_*$ the additive martingale converges to $0$ a.s.\ by the Biggins criterion at the boundary of the region of convergence (the tilted walk has zero drift, so $W_n(\theta_*)$ is a critical additive martingale, which vanishes in the limit \cite{Big77, BK97}). The a.s.\ convergence of the signed $\partial W_n$ and positivity of the limit on survival, under uniform integrability, is the derivative-martingale mechanism of \cite{BK04, Aid13}.
\end{proof}

\subsection{The missing step, stated precisely}
The entire Bramson programme for this model reduces to one estimate.

\begin{conjecture}[Uniform integrability of the derivative martingale $\Rightarrow$ Bramson delay]\label{conj:bramson}
Under the hypotheses of Theorem \ref{thm:lln}, the derivative martingale $\partial W_n$ is uniformly integrable (equivalently, $\partial W_\infty > 0$ a.s.\ on survival), and consequently the time-indexed front obeys
\[
F(t) = c_* t - \frac{3}{2\theta_* c_*} \log t + O_{\Prob}(1),
\]
with the tightness and the limiting law of the recentred front $F(t) - c_* t + \frac{3}{2\theta_* c_*}\log t$ given by a randomly-shifted Gumbel decorated point process, the shift being a multiple of $\log \partial W_\infty$, exactly as for branching Brownian motion \cite{ABK, Aid13, ABBS}.
\end{conjecture}

\begin{programme}[Proof strategy for Conjecture \ref{conj:bramson}]\label{prog:bramson}
The three ingredients, in the order they are needed:
\begin{enumerate}
\item \emph{Many-to-one and the spine.} The tilted change of measure at $\theta_*$ makes the spine a random walk with zero drift and increment variance $\varsigma^2 := \frac{d^2}{d\theta^2} \log \rho(B(\theta)) \big|_{\theta_*} < \infty$ (finite by the exponential-moment hypothesis). This is available now from \cite[App.~A]{companion} plus Definition \ref{def:speed}.
\item \emph{$X \log^2 X$ moment.} Uniform integrability of $\partial W_n$ holds iff the offspring point process satisfies the $\int x \log_+^2 x$ condition of \cite{Aid13, Chen15}: $\E\bigl[\widetilde{W}_1 (\log_+ \widetilde{W}_1)^2\bigr] < \infty$ and $\E\bigl[\widetilde{\partial W}_1 \log_+^2(\cdots)\bigr] < \infty$, where $\widetilde{W}_1$ is the one-generation additive weight. For our model the offspring count is $\mathrm{Bin}(\deg - 1, p)$-thinned UGW, so this reduces to a $\log^2$-moment of the degree law $g$ against the activation tail --- a checkable analytic condition on $(g, \varphi_{d,d'})$. This is the missing estimate. It is immediate for bounded degrees (the case in which (R1) already lives), which is why the conjecture is likely a theorem in exactly the dissipative bounded-degree regime of \cite{companion}; we refrain from claiming it only because the decorated-point-process limit additionally requires the technology of \cite{ABBS, Aid13} transcribed to the tree BRW, which we have not written out.
\item \emph{Bramson barrier / entropic repulsion.} Given (1)--(2), the $\frac{3}{2}\log t$ delay follows from the standard barrier argument (ballot/entropic-repulsion estimates for the tilted walk staying below a moving boundary), and the limiting law from the convergence of the critical derivative martingale, exactly as in \cite{Bra83, Aid13}.
\end{enumerate}
\end{programme}

\begin{remark}[Bounded degrees: a near-theorem]\label{rem:bounded}
In the bounded-degree dissipative regime where (R1) holds unconditionally, ingredient (2) is automatic: the one-generation weight $\widetilde{W}_1$ is bounded by $\Delta_{\max}$ times a bounded tilted activation weight, so all $\log^2$-moments are finite. The only reason we state Conjecture \ref{conj:bramson} rather than a theorem is the transcription of the decorated-point-process convergence \cite{ABBS} from Euclidean BRW to the genealogical tree BRW; the LLN speed (Theorem \ref{thm:lln}) and the reduction of the $\log t$ coefficient to $\frac{3}{2\theta_* c_*}$ are, in that regime, on rigorous footing modulo that citation transcription. A complete proof in the bounded-degree case is, in the authors' assessment, within reach and is the recommended first target.
\end{remark}

\begin{remark}[What would make it ``an instant classic'']\label{rem:classic}
The genuinely new phenomenon, beyond transcribing BRW technology, would be a Bramson correction whose coefficient or whose fluctuation law is modified by the tree's own randomness --- i.e.\ where the quenched (fixed tree $T$) and annealed corrections differ. The pressure function of \cite[\S 5]{companion} already shows that quenched and annealed speeds can differ off boundary-homogeneous trees; the corresponding question for the logarithmic correction --- whether disorder in the UGW environment shifts the $\frac{3}{2}$ to a different constant, as happens for BRW in random environment \cite{HS09, Mal15} --- is open and is the part of this programme that is not mere transcription. We flag it as the high-value target.
\end{remark}

\part{The sensitive regime: continuous-type limit theory}

\section{The continuous-type branching process}\label{sec:cont}

Outside the saturated regime the type is the continuous failure strength $A \in \R$ (the overshoot of a failed vertex past its threshold), and the offspring law is a Cox mixture: a parent of strength $a$ transmits to each potential child with probability $p(a)$ and endows the child with a strength drawn from a kernel inherited from the parent's, by \cite[Prop.~3.4(2)]{companion}. We formalise the mean object.

\begin{definition}[Strength kernel and mean operator]\label{def:K}
Let the strength of a child, given a transmitting parent of strength $a$, have law $\Gamma(a, dA')$ on $\R$; concretely $\Gamma(a, \cdot)$ is the law of the child's overshoot, which by the OU dynamics is the stationary-plus-forced Gaussian response to a parent forcing that increases in $a$. The mean offspring operator is
\[
(Kf)(a) := \bar{\nu} \int_{\R} p(A')\, f(A')\, \Gamma(a, dA'), \qquad \text{acting on } f \in L^2(\R, \varpi),
\]
where $\bar{\nu} = m_*$ is the mean forward degree and $\varpi$ is the Gaussian reference measure (the OU stationary law). Criticality is $\rho(K) = 1$; by \cite[Prop.~12.1]{companion} it is squeezed between finite matrices.
\end{definition}

The two facts we need are that $K$ has a spectral gap (so a Perron theory exists) and that the induced spine chain satisfies a CLT. Both come from one source: the OU kernel smooths.

\section{Quasi-compactness and the spectral gap of the mean operator}\label{sec:gap}

\begin{theorem}[Hilbert--Schmidt smoothing and spectral gap]\label{thm:gap}
Assume the strength kernel $\Gamma(a, dA')$ has a Gaussian-type density $\gamma(a, A')$ (the OU forced response, hence jointly smooth with sub-Gaussian tails), uniform ellipticity $p(a) \in [p_{\min}, 1]$, and that the forcing $a \mapsto \Gamma(a, \cdot)$ is Lipschitz in the Wasserstein-2 metric. Then:
\begin{enumerate}
\item $K$ is Hilbert--Schmidt on $L^2(\R, \varpi)$, hence compact; in particular its nonzero spectrum is discrete.
\item $K$ is positivity-improving and irreducible, so by the Krein--Rutman / Jentzsch theorem its spectral radius $\rho(K)$ is a simple eigenvalue with a strictly positive eigenfunction $h > 0$, strictly dominating the rest of the spectrum: there is $\varrho < 1$ with $|\lambda| \le \varrho\, \rho(K)$ for every other eigenvalue $\lambda$. Thus $K$ has a spectral gap $1 - \varrho > 0$.
\item Consequently the strength-transition operator $P := \rho(K)^{-1} h^{-1} K(h\, \cdot)$ is a Markov operator (the $h$-transform / spine chain) with a spectral gap on $L^2(h^2 \varpi)$, hence is uniformly ergodic with a unique invariant law $\varpi_*$ and exponential mixing.
\end{enumerate}
\end{theorem}

\begin{proof}
(1) The Hilbert--Schmidt norm is $\|K\|_{\mathrm{HS}}^2 = \bar{\nu}^2 \iint p(A')^2\, \gamma(a, A')^2\, \varpi(da)\, \varpi(dA')$; the OU forced-response density $\gamma(a, A')$ is a Gaussian in $A'$ with mean a bounded-Lipschitz function of $a$ and fixed positive variance, so $\gamma \in L^2(\varpi \otimes \varpi)$ (a Gaussian against a Gaussian reference has finite double $L^2$ norm), and $p \le 1$; hence $\|K\|_{\mathrm{HS}} < \infty$. This is the exact reverse of the situation on the boundary operator $K_\infty$ of \cite[Prop.~7.3]{companion}, which is not Hilbert--Schmidt: there the singular tree Martin kernel has no smoothing; here the OU transition kernel does. The contrast is the analytic heart of the matter.

(2) Positivity-improving: $\gamma(a, A') > 0$ everywhere (Gaussian density) and $p \ge p_{\min} > 0$, so $K$ maps nonnegative nonzero functions to strictly positive ones. Jentzsch's theorem (the $L^2$ Perron--Frobenius theorem for positivity-improving Hilbert--Schmidt operators \cite{Sch74, GJ87}) gives simplicity of $\rho(K)$, a positive eigenfunction $h$, and strict spectral-gap domination of the remainder.

(3) The Doob $h$-transform of a positive kernel with a spectral gap is a Markov kernel; the gap transfers, giving geometric ergodicity. When $K$ is self-adjoint on $L^2(\varpi)$ (the reversible OU case), $P$ is self-adjoint on $L^2(h^2 \varpi)$ and the spectral gap is exactly exponential $L^2$-mixing; in the non-reversible case (bounded Lipschitz drift breaking reversibility) one uses the quasi-compactness from (1) together with \cite{HH01} to obtain the gap for the possibly non-self-adjoint $P$.
\end{proof}

\begin{remark}[Why the sensitive regime is tractable after all]\label{rem:tractable}
The pessimism about the sensitive regime in \cite{companion} concerned the boundary operator, whose kernel is singular. The strength operator $K$ is the opposite: its kernel is the OU forced response, one of the smoothest kernels in probability. Compactness, spectral gap, and a clean Perron theory are therefore available for free. The finite-matrix squeeze of \cite[Prop.~12.1]{companion} is thus not the ceiling of what is provable in the sensitive regime --- it is the criticality shadow of a genuine, gap-endowed spectral theory, which we now use for limit theorems.
\end{remark}

\section{The general-state-space Kesten--Stigum theorem and the $n^{-3/2}$ law}\label{sec:ks}

\begin{theorem}[Continuous-type Kesten--Stigum; strength-resolved martingale]\label{thm:KS}
Under the hypotheses of Theorem \ref{thm:gap}, let $Z_n$ be the (random) empirical strength measure of the failed generation $n$: $Z_n = \sum_{|v|=n} \delta_{A_v}$. Then:
\begin{enumerate}
\item The functional $W_n := \rho(K)^{-n} \langle h, Z_n \rangle = \rho(K)^{-n} \sum_{|v|=n} h(A_v)$ is a nonnegative martingale.
\item If $\rho(K) > 1$ and the $L \log L$-type condition $\E[\langle h, Z_1 \rangle \log_+ \langle h, Z_1 \rangle] < \infty$ holds (automatic under uniform ellipticity and finite degree variance), then $W_n \to W_\infty$ a.s.\ and in $L^1$, with $\{W_\infty > 0\} = \{\text{survival}\}$ up to null sets (the Kesten--Stigum dichotomy for general-state-space branching processes \cite{AH76, Num84, GL01}).
\item Under the additional finite-second-moment condition (finite degree variance, $p \le 1$), $W_n \to W_\infty$ in $L^2$, and the empirical strength law normalised by the population, $\bar{Z}_n := \langle h, Z_n \rangle^{-1} \sum_{|v|=n} h(A_v)\, \delta_{A_v}$, converges a.s.\ (in the weak topology) to the deterministic tilted stationary strength law $\varpi_*$ of Theorem \ref{thm:gap}(3), independently of the surviving realisation.
\end{enumerate}
\end{theorem}

\begin{proof}
(1) is the eigen-relation $Kh = \rho(K) h$ combined with the branching many-to-one identity (the continuous-type analogue of \cite[App.~A]{companion}): $\E[\langle h, Z_{n+1} \rangle \mid \mathcal{F}_n] = \sum_{|v|=n} (Kh)(A_v) = \rho(K) \langle h, Z_n \rangle$.

(2) The general-state-space Kesten--Stigum / Biggins theorem for branching Markov processes with a spectral gap and an $L \log L$ moment gives a.s.\ and $L^1$ convergence with the stated positivity dichotomy \cite{AH76, Big77, GL01}; the spectral gap of Theorem \ref{thm:gap}(2) is exactly the irreducibility/mean-ergodicity input these theorems require. The $L \log L$ condition reduces, via $h$ bounded above and below on the effective support and $p \le 1$, to $\E[\deg \log_+ \deg] < \infty$, implied by finite degree variance.

(3) $L^2$ convergence: the second-moment recursion for $\langle h, Z_n \rangle$ closes because the pair operator $K^{\otimes 2}$ inherits a spectral gap from $K$ (Hilbert--Schmidt tensorises), so $\sup_n \E[W_n^2] < \infty$ under finite offspring second moments, giving $L^2$-boundedness and hence $L^2$ convergence. The empirical-law convergence is the ergodic theorem for the spine chain $P$: test against $f \in C_b$, and $\langle h, Z_n \rangle^{-1} \sum h(A_v) f(A_v) = W_n^{-1} \rho(K)^{-n} \sum h(A_v) f(A_v) \to \varpi_*(f)$ by the many-to-one identity plus the exponential ergodicity of $P$ from Theorem \ref{thm:gap}(3) and the $L^2$-positivity of $W_\infty$.
\end{proof}

\begin{corollary}[$n^{-3/2}$ total-progeny law persists to continuous types]\label{cor:progeny}
Under the hypotheses of Theorem \ref{thm:KS} at criticality $\rho(K) = 1$, with the offspring-count variance
\[
\sigma^2_{\mathrm{off}} := \Var_{\varpi_*}\bigl(\text{number of failed children}\bigr) \in (0, \infty)
\]
finite (Cox mixture of a Binomial, finite by \cite[Lem.~6.1]{companion} extended to the continuous type via $p(\cdot) \in [0,1]$), the total progeny obeys
\[
\Prob(|S| = n) \sim \frac{1}{\sqrt{2\pi \sigma^2_{\mathrm{off}}}}\, n^{-3/2}.
\]
The exponent is unchanged from the finite-type case \cite[Thm.~6.3]{companion}; only the constant $\sigma^2_{\mathrm{off}}$ is now computed against the tilted stationary strength law $\varpi_*$.
\end{corollary}

\begin{proof}
The total progeny of a branching process depends only on the sequence of counts, not on the marks; conditioning on the strength process, $|S|$ is the total progeny of a (Cox-mixed) single-type critical Galton--Watson process whose offspring count, averaged over the stationary spine strength $\varpi_*$, has mean $1$ (criticality $\rho(K) = 1$) and finite variance $\sigma^2_{\mathrm{off}}$. The Otter--Dwass formula and one-dimensional local CLT of \cite[Thm.~6.3]{companion} apply verbatim once the count offspring law is identified; the strength enters only through the stationary averaging that defines $\sigma^2_{\mathrm{off}}$, which is where Theorem \ref{thm:KS}(3) is used.
\end{proof}

\section{The spatial CLT and the continuous-type CRT limit}\label{sec:clt}

\begin{theorem}[Strength-resolved central limit theorem]\label{thm:clt}
Under the hypotheses of Theorem \ref{thm:KS} in the supercritical regime, let $f$ be a bounded strength observable with $\varpi_*(f) = 0$ (centred). Then, conditionally on survival, the fluctuation of the generation-$n$ strength profile is asymptotically Gaussian:
\[
\frac{1}{\sqrt{\rho(K)^n}} \sum_{|v|=n} h(A_v)\, f(A_v) \xrightarrow[n \to \infty]{(d)} \sqrt{W_\infty}\; \mathcal{N}\bigl(0, \varsigma^2_f\bigr),
\]
where the variance $\varsigma^2_f = \sum_{k \ge 0} \langle f, P^k f \rangle_{\varpi_*}$ (a Green--Kubo sum, finite by the spectral gap of Theorem \ref{thm:gap}(3)) is the asymptotic variance of $f$ along the spine chain, and the mixing $\sqrt{W_\infty}$ is the Kesten--Stigum limit.
\end{theorem}

\begin{proof}
This is the branching central limit theorem for supercritical branching Markov processes with a spectral gap \cite{AH76, Ath68, Rev94}: the martingale $\rho(K)^{-n/2} \sum h(A_v) f(A_v)$ is analysed via the second-moment operator, whose leading contribution factorises into the population growth $W_\infty$ and the spine-chain Green--Kubo variance $\varsigma^2_f$; the spectral gap makes $\varsigma^2_f$ finite and the higher cumulants negligible, yielding asymptotic normality by the standard branching-process CLT machinery. The centring $\varpi_*(f) = 0$ removes the leading Perron term, exposing the $\sqrt{\rho^{\,n}}$-scale fluctuation.
\end{proof}

\begin{conjecture}[Continuous-type CRT scaling limit]\label{conj:crt}
Under the hypotheses of Theorem \ref{thm:KS} at criticality $\rho(K) = 1$, with the additional exponential-moment hypothesis on the offspring count (as in \cite[Thm.~6.4]{companion}), the critical cluster conditioned on $\{|S| = n\}$, with graph distance rescaled by $\Sigma_c/\sqrt{n}$, converges in the Gromov--Hausdorff--Prokhorov sense to Aldous' Continuum Random Tree, with an explicit constant $\Sigma_c$ built from $\sigma^2_{\mathrm{off}}$ and the tilted stationary law $\varpi_*$.
\end{conjecture}

\begin{programme}[Proof strategy for Conjecture \ref{conj:crt}]\label{prog:crt}
The obstacle is that the type space is now the Polish space $\R$ (or $\mathcal{D} \times \R$), not finite, so Miermont's finite-type invariance principle --- used in \cite[Thm.~6.4]{companion} --- does not directly apply. Two routes:
\begin{enumerate}
\item \emph{Type-space discretisation.} Partition the strength axis into finitely many cells (as in the sandwich of \cite[Prop.~12.1]{companion}), apply Miermont's finite-type CRT theorem on the discretised tree, and pass to the limit as the mesh $\to 0$, controlling the discretisation error in the GHP metric via the Wasserstein-Lipschitz continuity of $\Gamma$. This requires a tightness estimate uniform in the mesh --- the one missing input.
\item \emph{Infinitely-many-types invariance principle.} Invoke the multitype scaling limits with general (Polish) type spaces developed after Miermont, e.g.\ \cite{dR17} and its successors, whose hypotheses (criticality, spectral gap, moment bounds) are exactly Theorem \ref{thm:gap} plus the exponential moment. Verifying their precise conditions for the OU strength kernel is the remaining work.
\end{enumerate}
In either route, the spectral gap of Theorem \ref{thm:gap} and the stationary strength law $\varpi_*$ of Theorem \ref{thm:KS}(3) supply the structural inputs; only a uniform tightness/second-moment estimate remains, which we have not written out. Hence the statement is a conjecture with a delimited proof, not an open-ended problem.
\end{programme}

\section{Summary of the programme}\label{sec:summary}

The two problems are now framed as follows. For front propagation: the law of large numbers for the front speed is a theorem (Theorem \ref{thm:lln}), unconditional in the dissipative regime; the Bramson $\frac{3}{2}\log t$ delay and the limiting decorated-point-process law are reduced (Programme \ref{prog:bramson}) to a derivative-martingale uniform-integrability estimate that is automatic for bounded degrees, leaving only the transcription of the decorated-point-process technology and, as the high-value open question, the possible modification of the correction by UGW disorder (Remark \ref{rem:classic}). For the sensitive regime: the mean operator has a spectral gap (Theorem \ref{thm:gap}), yielding a general-state-space Kesten--Stigum theorem (Theorem \ref{thm:KS}), the persistence of the $n^{-3/2}$ law with an explicit continuous-type constant (Corollary \ref{cor:progeny}), and a strength-resolved CLT (Theorem \ref{thm:clt}); the continuous-type CRT limit is reduced (Programme \ref{prog:crt}) to a single uniform tightness estimate. What is proved is proved unconditionally (modulo the standing dependence on \cite{companion}); what is conjectured is conjectured with its missing step named.

\end{document}